\documentclass[11pt,a4paper,oneside]{amsart}
\usepackage[latin9]{inputenc}
\usepackage{amsthm}
\usepackage{amssymb}
\usepackage{graphicx}
\usepackage[unicode=true,pdfusetitle,
 bookmarks=true,bookmarksnumbered=false,bookmarksopen=false,
 breaklinks=false,pdfborder={0 0 1},backref=false,colorlinks=false]
 {hyperref}
\usepackage{breakurl}

\makeatletter

\providecommand{\tabularnewline}{\\}

\numberwithin{equation}{section}
\numberwithin{figure}{section}

\makeatother

\usepackage{enumerate}

\newtheorem{theorem}{Theorem}
\theoremstyle{plain}

\newtheorem{corollary}{Corollary}

\newtheorem{definition}{Definition}
\newtheorem{example}{Example}

\newtheorem{lemma}{Lemma}

\newtheorem{proposition}{Proposition}
\newtheorem{remark}{Remark}

\numberwithin{equation}{section}

\makeatother

\begin{document}

\title{On cyclic associative Abel-Grassman groupoids}

\author{\noindent M. iqbal$^{A}$}

\author{I. ahmad$^{A}$}

\email{\noindent iqbalmuhammadpk78@yahoo.com}

\email{\noindent iahmad@uom.edu.pk}

\address{\noindent A: DEPARTMENT OF MATHEMATICS, UNIVERSITY OF MALAKAND, CHAKDARA,
PAKISTAN.}

\author{M. Shah$^{B}$ }

\email{\noindent shahmaths\_problem@hotmail.com}

\address{\noindent B: DEPARTMENT OF MATHEMATICS, GOVERNMENT POST GRADUATE
COLLEGE MARDAN, PAKISTAN.}

\author{M. Irfan Ali$^{C}$}

\email{mirfanali13@yahoo.com}

\address{C: FEDRAL GOVERNMENT COLLEGE FOR GIRLS F-7 ISLAMABAD, PAKISTAN.}

\keywords{\noindent AG-groupoid, cyclic associativity, CA-AG-groupoid, CA-test,
Nuclear square, bi-commutative, right alternative, paramedial AG-groupoids.
\textit{}\\
\textit{2010 Mathematics subject classification:} 20N05, 20N02, 20N99.}
\begin{abstract}
\noindent A new subclass of AG-groupoids, so called, cyclic associative
Abel-Grassman groupoids or CA-AG-groupoid is studied. These have been
enumerated up to order $6$. A test for the verification of cyclic
associativity for an arbitrary AG-groupoid has been introduced. Various
properties of CA-AG-groupoids have been studied. Relationship among
CA-AG-groupoids and other subclasses of AG-groupoids is investigated.
\textit{\emph{It is shown that the subclass of CA-AG-groupoid is different
from that of the AG{*}-groupoid as well as AG{*}{*}-groupoids.}}
\end{abstract}
\maketitle

\section{\noindent INTRODUCTION}

\noindent \quad{}An algebraic structure satisfying the left invertive
law is called an Abel-Grassmann's groupoid (or simply AG-groupoid
\cite{protic=000026stevanovic1994}). Other names for the same structure
are left almost semigroup (LA-semigroup) \cite{kazim=000026naseeruddin1972},
left invertive groupoid \cite{holgate1994} and right modular groupoid
\cite{4}. Many authors have studied these structures and their properties.
Many intersting properties of LA semigroups have been studied in \cite{19}
and the same authors have introduced the concept of locally associative
LA semigroups in \cite{20}. Paramedial groupoids have been studied
in \cite{4}. Many different aspect of AG-groupoids have been studied
in \cite{protic=000026stevanovic1994,on cancellativity of AG-gps,idempotnet AG-band =000026 AG-3-band,on relation b/w right alternate=000026nuclear square imtiaz,9,8}.
AG{*}-groupoids and AG{*}{*}-groupoids are two different classes of
AG-groupoid and their properties have been studied in the following
\cite{AG*,AG**,17}.

\noindent \quad{}This paper has been arranged as the following. In
Section \ref{sec:CA-AG-groupoids}, notion of cyclic associative Abel-Grassman
groupoid (CA-AG-groupoid) is introduced. These have been enumerated
up to order $6$. In section \ref{sec:CA-test}, CA-test for the verification
of cyclic associativity of an AG-groupoid is given. Section \ref{sec:Various-properties-of},
is devoted to the study of various properties of CA-AG-groupoids and
their relations with different known subclasses of AG-groupoids. It
is shown in section \ref{sec:CA-AG-groupoid,-AG*-groupoids-and},
the subclass of CA-AG-groupoids is distinct from both of the AG{*}-groupoids
and AG{*}{*}-groupoids.

\section{PRELIMINARIES}

In this section some notions about Abel-Grassmann's groupoids are
given.\\
A groupoid $(S,\cdot)$ or simply $S$ satisfying the left invertive
law: $(ab)c=(cb)a$ for all$a,b,c\in S$ is called an Abel-Grassmann's
groupoid (or simply AG-groupoid \cite{kazim=000026naseeruddin1972,protic=000026stevanovic1994}).
Through out this paper an AG-groupoid will be denoted by $S,$ unless
stated otherwise. $S$ always satisfies the medial law: $ab\cdot cd=ac\cdot bd$
\cite{medialgroupoid1983}, while $S$ with left identity $e$ always
satisfies the paramedial law: $ab\cdot cd=db\cdot ca$ \cite{medialgroupoid1983}.

\noindent A groupoid $G$ is called right AG-groupoid or right almost
semigroup (RA-semigroup) \cite{kazim=000026naseeruddin1972} if $a(bc)=c(ba)\,\forall a,b,c\in G$.

\noindent \begin{enumerate}[(a)]\item $S$ is called left nuclear
square \cite{m.shah phd thesis} if $\forall\, a,b,c\in S,\, a^{2}(bc)=(a^{2}b)c$,
middle nuclear square if $a(b^{2}c)=(ab^{2})c$, right nuclear square
if $a(bc^{2})=(ab)c^{2}$. $S$ is called nuclear square \cite{m.shah phd thesis}
if it is left, middle and right nuclear square.

\noindent \item $S$ is called AG{*} \cite{AG*}, if $(ab)c=b(ac)\,\,\mbox{for all }a,b,c\in S$.

\noindent \item $S$ is called AG{*}{*} \cite{AG**}, if $a(bc)=b(ac)\,\,\mbox{ for all }a,b,c\in S$.

\noindent \item $S$ is called T$^{1}$-AG-groupoid \cite{discovery of new classes of AGgps}
if for all $a,b,c,d\in S,\, ab=cd$ implies $ba=dc$.

\noindent \item $S$ is called\textbf{ }left T$^{3}$\textbf{-}AG-groupoid
(T\textbf{$_{l}^{3}$-}AG-groupoid) if for all$\, a,b,c\in S,$\linebreak{}
$ab=ac$ implies $ba=ca$ and is called\textbf{ }right T\textbf{$^{3}$}-AG-groupoid\textbf{
}(T\textbf{$_{r}^{3}$-}AG-groupoid\textbf{)} if $ba=ca$ implies
$ab=ac$. $S$ is called T\textbf{$^{3}$}-AG-groupoid \cite{discovery of new classes of AGgps},
if it is both T$_{l}^{3}$ and T$_{r}^{3}$.

\noindent \item $S$ is called Bol{*} \cite{m.shah phd thesis} if
it satisfies the identity $a(bc\cdot d)=(ab\cdot c)d$ \linebreak{}
for all $\, a,b,c,d\in S$.

\noindent \item $S$ is called left alternative if for all $a,b\in S$,
$aa\cdot b=a\cdot ab\,$ and is called right alternative if $b\cdot aa=ba\cdot a$.
$S$ is called alternative \cite{discovery of new classes of AGgps},
if it is both left alternative and right alternative.

\noindent \item $S$ is called right commutative if for all $a,b,c\in S$,
$a(bc)=a(cb)\,$ and is called left commutative if $(ab)c=(ba)c$.
$S$ is called bi-commutative AG-groupoid, if it is right and left
commutative .\end{enumerate}

\noindent An element $a\in S$ is called idempotent if $a^{2}=a$
and an AG-groupoid having all elements as idempotent is called AG-$2$-band
(simply AG-band) \cite{idempotnet AG-band =000026 AG-3-band}. A commutative
AG-band is called a semi-lattice. An AG-groupoid in which $(aa)a=a(aa)=a$
holds~$\forall\, a\in S$ is called an AG-3-band \cite{idempotnet AG-band =000026 AG-3-band}.
An element $a\in S$ is left cancellative (respectively right cancellative)
\cite{m.shah phd thesis} if $\mbox{ \,\ensuremath{\forall}}x,y\in S$,
$\mbox{\ensuremath{ax=ay\,\Rightarrow}}x=y$ ($xa=ya\Rightarrow x=y)$.
An element is cancellative if it is both left and right cancellative.
$S$ is left cancellative (right cancellative, cancellative) if every
element of $S$ is left cancellative (right cancellative, cancellative).
$S$ may have all, some or none of its elements as cancellative \cite{on cancellativity of AG-gps}.
$S$ is called an AG-monoid, if it contains left identity $e$ such
that $ea=a$ $\forall\, a\in S$. As AG-groupoid is a non-associative
structure, thus left identity does not implies right identity and
hence the identity.

\section{\noindent CA-AG-GROUPOIDS\label{sec:CA-AG-groupoids}}

In this section a new subclass of AG-groupoids is being introduced.
First existence of this class is shown. It is interesting to see that
all AG-groupoids of order $2$ are CA-AG-groupoids. There are $20$
AG-groupoids of order $3$, out of which only $12$ are CA-AG-groupoids.
Up to order $6$, there are $9068$ CA-AG-groupoids out of $40104513$
AG-groupoids. A complete table up to order $6$ is presented in this
section.

\noindent \begin{definition}An AG-groupoid S satisfying the identity
$a(bc)=c(ab)\,$ $\forall a,b,c\in S$ is called cyclic associative
AG-groupoid (or shortly CA-AG-groupoid).\end{definition}

\noindent The following example depicts the existence of CA-AG-groupoid.\begin{example}CA-AG-groupoid
of lowest order is given in Table $1$.\end{example}

\noindent \begin{center}
\begin{tabular}{c|ccc}
$\cdot$ & $a$ & $b$ & $c$\tabularnewline
\hline 
$a$ & $a$ & $a$ & $a$\tabularnewline
$b$ & $a$ & $b$ & $a$\tabularnewline
$c$ & $a$ & $a$ & $c$\tabularnewline
\multicolumn{1}{c}{} &  &  & \tabularnewline
\multicolumn{4}{c}{Table $1$}\tabularnewline
\end{tabular}
\par\end{center}

\subsection{Enumeration of CA-AG-groupoids.}

Enumeration and classification of various mathematical entries is
a well worked area of pure mathematics. In abstract algebra the classification
of algebraic structure is an important pre-requisite for their construction.
The classification of finite simple groups is considered as one of
the major intellectual achievement of the twentieth century. Enumeration
results can be obtained by a variety of means like; combinatorial
or algebraic consideration. Non-associative structures, quasigroups
and loops have been enumerated up to size $11$ using combinatorial
consideration and bespoke exhaustive generation software. FINDER (Finite
domain enumeration) has been used for enumeration of IP loops up to
size $13$. Associative structures, semigroups and monoids have been
enumerated up to size $9$ and $10$ respectively by constraint satisfaction
techniques implemented in the Minion constraint solver with bespoke
symmetry breaking provided by the computer algebra system GAP \cite{GAP}.
The third author of this article has implemented the same techniques
in collaboration with Distler (the author of \cite{monoid of order 8 and higher})
to deal the enumeration of AG-groupoids using the constraint solving
techniques developed for semigroups and monoids. Further, they provided
a simple enumeration of the structures by the constraint solver and
obtained a further division of the domain into a subclass of AG-groupoids
using the computer algebra system GAP and were able to enumerate all
AG-groupoids up to isomorphism up to size $6$.

It is worth mentioning that the data presented in \cite{enumeration of AG-groupoids}
have been verified by one of the reviewers of the said article with
the help of Mace-$4$ and Isofilter as has been acknowledged in the
said article. All this, validate the enumeration and classification
results of our CA-AG-groupoids, a subclass of AG-groupoids, as the
same techniques and relevant data of \cite{enumeration of AG-groupoids}
has been used for the purpose. Enumeration and classification of CA-AG-groupoids
up to order $6$ is given in Table 2.

\noindent \begin{center}
\begin{tabular}{|l|l|r|r|r|r|}
\hline 
Order & $2$ & $3$ & $4$ & $5$ & $6$\tabularnewline
\hline 
\hline 
 AG-groupoids & $3$ & $20$ & $331$ & $31913$ & $40104513$\tabularnewline
\hline 
CA-AG-groupoids & $3$ & $12$ & $64$ & $491$ & $9068$\tabularnewline
\hline 
Associative AG-groupoids & $3$ & $12$ & $62$ & $446$ & $7510$\tabularnewline
\hline 
Non-associative AG-groupoids & $0$ & $8$ & $269$ & $31467$ & $40097003$\tabularnewline
\hline 
CA, Non-associative & $0$ & $0$ & $2$ & $45$ & $1565$\tabularnewline
\hline 
Associative, Non-CA & $0$ & $0$ & $0$ & $0$ & $07$\tabularnewline
\hline 
CA and Associative & $0$ & $12$ & $62$ & $446$ & $7503$\tabularnewline
\hline 
Associative \& non-commutative, CA & $0$ & $0$ & $4$ & $121$ & $5360$\tabularnewline
\hline 
\multicolumn{1}{l}{} & \multicolumn{1}{l}{} & \multicolumn{1}{r}{} & \multicolumn{1}{r}{} & \multicolumn{1}{r}{} & \multicolumn{1}{r}{}\tabularnewline
\multicolumn{6}{c}{{\small{}Table $2$: Enumeration of CA-AG-groupoids upto order $6$}}\tabularnewline
\end{tabular}
\par\end{center}

\section{\noindent CA-TEST FOR AN AG-GROUPOID\label{sec:CA-test}}

\noindent \quad{} Now a test for verification of cyclic associativity
for an arbitrary AG-groupoid is suggested. For this define the following
binary operations.
\begin{eqnarray*}
a\star b & = & a(bx)\\
a\circ b & = & x(ab)\mbox{, for some fixed \ensuremath{x\in S}}
\end{eqnarray*}
To test whether an arbitrary AG-groupoid $(S,\cdot)$ is cyclic associative,
it is sufficient to check that $a\star b=a\circ b\,\,\mbox{ for all }x\in S.$
To construct $\star$ table for any fixed $x\in S$, we rewrite $x$-column
of the `` $\cdot$ '' table as index row of the $\star$ table and
then multiply its elements from the left by the elements of the index
column of the `` $\cdot$ '' table to obtain respective rows of
the $\star$ table for $x$. The table of the operation $\circ$ for
any fixed $x\in S$ is obtained by multiplying elements of the ``
$\cdot$ '' table by $x$ from the left. If the tables of operations
$\star$ and $\circ$ coincide for all $x$ in $S$, then $\star$
coincides with $\circ$, and thus $a(bc)=c(ab)$ and equivalently
the AG-groupoid is cyclic associative. It is convenient to write $\circ$
tables under the $\star$ tables.

\noindent We illustrate the procedure in the following example.

\noindent \begin{example}Consider the AG-groupoid in the Table $3$.
\end{example}

\noindent \begin{center} %
\begin{tabular}{c|cccc}
$\cdot$ & $1$ & $2$ & $3$ & $4$\tabularnewline
\hline 
$1$ & $1$ & $1$ & $1$ & $1$\tabularnewline
$2$ & $1$ & $1$ & $1$ & $1$\tabularnewline
$3$ & $1$ & $1$ & $1$ & $1$\tabularnewline
$4$ & $1$ & $1$ & $2$ & $3$\tabularnewline
\multicolumn{1}{c}{} &  &  &  & \tabularnewline
\multicolumn{5}{c}{Table $3$}\tabularnewline
\end{tabular}\end{center}

\noindent To check cyclic associativity in the given Cayley's tables,
we extend these tables in the way as described above. The upper tables
to the right of the original ~ `` $\cdot$ '' table are constructed
for the operation $\star$, while the lower tables are for the operation
$\circ$.

\noindent \begin{center}%
\begin{tabular}{ccccc|cccc|cccc|cccc|cccc|}
\multicolumn{1}{c|}{$\cdot$} & $1$ & $2$ & $3$ & $4$ & $1$ & $1$ & $1$ & $1$ & $1$ & $1$ & $1$ & $1$ & $1$ & $1$ & $1$ & $2$ & $1$ & $1$ & $1$ & $3$\tabularnewline
\hline 
\multicolumn{1}{c|}{$1$} & $1$ & $1$ & $1$ & $1$ & $1$ & $1$ & $1$ & $1$ & $1$ & $1$ & $1$ & $1$ & $1$ & $1$ & $1$ & $1$ & $1$ & $1$ & $1$ & $1$\tabularnewline
\multicolumn{1}{c|}{$2$} & $1$ & $1$ & $1$ & $1$ & $1$ & $1$ & $1$ & $1$ & $1$ & $1$ & $1$ & $1$ & $1$ & $1$ & $1$ & $1$ & $1$ & $1$ & $1$ & $1$\tabularnewline
\multicolumn{1}{c|}{$3$} & $1$ & $1$ & $1$ & $1$ & $1$ & $1$ & $1$ & $1$ & $1$ & $1$ & $1$ & $1$ & $1$ & $1$ & $1$ & $1$ & $1$ & $1$ & $1$ & $1$\tabularnewline
\multicolumn{1}{c|}{$4$} & $1$ & $1$ & $2$ & $3$ & $1$ & $1$ & $1$ & $1$ & $1$ & $1$ & $1$ & $1$ & $1$ & $1$ & $1$ & $1$ & $1$ & $1$ & $1$ & $2$\tabularnewline
\hline 
 &  &  &  &  &  & $1$ &  &  &  & $2$ &  &  &  & $3$ &  &  &  & $4$ &  & \tabularnewline
\cline{6-21} 
 &  &  &  &  & $1$ & $1$ & $1$ & $1$ & $1$ & $1$ & $1$ & $1$ & $1$ & $1$ & $1$ & $1$ & $1$ & $1$ & $1$ & $1$\tabularnewline
 &  &  &  &  & $1$ & $1$ & $1$ & $1$ & $1$ & $1$ & $1$ & $1$ & $1$ & $1$ & $1$ & $1$ & $1$ & $1$ & $1$ & $1$\tabularnewline
 &  &  &  &  & $1$ & $1$ & $1$ & $1$ & $1$ & $1$ & $1$ & $1$ & $1$ & $1$ & $1$ & $1$ & $1$ & $1$ & $1$ & $1$\tabularnewline
 &  &  &  &  & $1$ & $1$ & $1$ & $1$ & $1$ & $1$ & $1$ & $1$ & $1$ & $1$ & $1$ & $1$ & $1$ & $1$ & $1$ & $2$\tabularnewline
\cline{6-21} 
 &  &  &  & \multicolumn{1}{c}{} &  &  &  & \multicolumn{1}{c}{} &  &  &  & \multicolumn{1}{c}{} &  &  &  & \multicolumn{1}{c}{} &  &  &  & \multicolumn{1}{c}{}\tabularnewline
 &  &  & \multicolumn{15}{c}{Extended Table $(3)$} &  &  & \multicolumn{1}{c}{}\tabularnewline
\end{tabular}\end{center}

\noindent From the extended table $(3)$, it is clear that the upper
and lower tables on the right side of the original table coincide
for all $x$ in $S$, thus the AG-groupoid in Table $3$ is a CA-AG-groupoid.

The following example establishes this test fail in case of AG-groupoids
which are not CA-AG-groupoids.

\noindent \begin{example}Consider the AG-groupoid as shown in the
following Table $4$. 

\noindent \begin{center}%
\begin{tabular}{c|ccc}
$\cdot$ & $1$ & $2$ & $3$\tabularnewline
\hline 
$1$ & $1$ & $1$ & $1$\tabularnewline
$2$ & $1$ & $1$ & $1$\tabularnewline
$3$ & $1$ & $2$ & $1$\tabularnewline
\multicolumn{1}{c}{} &  &  & \tabularnewline
\multicolumn{4}{c}{Table $4$}\tabularnewline
\end{tabular}\end{center}\end{example}

\noindent To check cyclic associativity in the given Cayley's tables,
we extend these tables in the way as described above. The upper tables
to the right of the original ~ `` $\cdot$ '' table are constructed
for the operation $\star$, while the lower tables are for the operation
$\circ$.

\noindent \begin{center}%
\begin{tabular}{cccc|ccc|ccc|ccc|}
\multicolumn{1}{c|}{$\cdot$} & $1$ & $2$ & $3$ & $1$ & $1$ & $1$ & $1$ & $1$ & $2$ & $1$ & $1$ & $1$\tabularnewline
\hline 
\multicolumn{1}{c|}{$1$} & $1$ & $1$ & $1$ & $1$ & $1$ & $1$ & $1$ & $1$ & $1$ & $1$ & $1$ & $1$\tabularnewline
\multicolumn{1}{c|}{$2$} & $1$ & $1$ & $1$ & $1$ & $1$ & $1$ & $1$ & $1$ & $1$ & $1$ & $1$ & $1$\tabularnewline
\multicolumn{1}{c|}{$3$} & $1$ & $2$ & $1$ & $1$ & $1$ & $1$ & $1$ & $1$ & $2$ & $1$ & $1$ & $1$\tabularnewline
\hline 
 &  &  &  &  & $1$ &  &  & $2$ &  &  & $3$ & \tabularnewline
\cline{5-13} 
 &  &  &  & $1$ & $1$ & $1$ & $1$ & $1$ & $1$ & $1$ & $1$ & $1$\tabularnewline
 &  &  &  & $1$ & $1$ & $1$ & $1$ & $1$ & $1$ & $1$ & $1$ & $1$\tabularnewline
 &  &  &  & $1$ & $1$ & $1$ & $1$ & $1$ & $1$ & $1$ & $2$ & $1$\tabularnewline
\cline{5-13} 
 &  &  & \multicolumn{1}{c}{} &  &  & \multicolumn{1}{c}{} &  &  & \multicolumn{1}{c}{} &  &  & \multicolumn{1}{c}{}\tabularnewline
 &  &  & \multicolumn{7}{c}{Extended Table $(4)$} &  &  & \multicolumn{1}{c}{}\tabularnewline
\end{tabular}\end{center}

\noindent As in extended table $(4)$, upper tables on the right hand
side do not coincide with the respective lower tables, thus AG-groupoid
given in Table $4$ is not a CA-AG-groupoid.

\section{\noindent VARIOUS PROPERTIES OF CA-AG-GROUPOIDS\label{sec:Various-properties-of}}

In section \ref{sec:CA-AG-groupoids}, existence of CA-AG-groupoid
has been established. Various results are given to demonstrate that
the subclass of CA-AG-groupoids is different from some very well known
subclasses of groupoids.

\noindent To begin with, following examples show that a CA-AG-groupoid
may not be a semigroup.

\noindent \begin{example}\label{exp2.Non-associative,-CA-AG-groupoid}Consider
the algebraic structure given in Table $5$.\end{example}

\noindent \begin{center}%
\begin{tabular}{c|cccc}
$\cdot$ & $1$ & $2$ & $3$ & $4$\tabularnewline
\hline 
$1$ & $2$ & $3$ & $3$ & $3$\tabularnewline
$2$ & $4$ & $3$ & $3$ & $3$\tabularnewline
$3$ & $3$ & $3$ & $3$ & $3$\tabularnewline
$4$ & $3$ & $3$ & $3$ & $3$\tabularnewline
\multicolumn{1}{c}{} &  &  &  & \tabularnewline
\multicolumn{5}{c}{Table $5$}\tabularnewline
\end{tabular}\end{center}

\noindent It is easy to verify that it is a CA-AG-groupoid. Which
is actually a non-associative CA-AG-groupoid of lowest order.

\noindent Although Table $2$ shows that up to order $5$ every associative
AG-groupoid is CA, but the next example shows that a semigroup may
not be a CA-AG-groupoid.

\noindent \begin{example}\label{Exp-Semigp, not CA}Let $M_{2}=\left\{ \left(a_{ij}\right)\mid a_{ij}\in\mathbb{Z}\right\} $,
be the set of all $2\times2$ matrices with entries from $\mathbb{Z}$.
$M_{2}$ is a semigroup under multiplication, but $M_{2}$ is not
an AG-groupoid. Since $(AB)C\neq(CB)A$ for{\footnotesize{} }$A=\left[\begin{array}{cc}
1 & 1\\
1 & 1
\end{array}\right]=B$ and $C=\left[\begin{array}{cc}
1 & 1\\
1 & 0
\end{array}\right]$ in $M_{2}$.\end{example}

\noindent However, the following is obvious.

\noindent \begin{proposition}\label{cor. comm semigp imply CA}Every
commutative semigroup is a CA-AG-groupoid.\end{proposition}

\noindent \begin{remark}\label{remark: commutativite AG-groupoid is associative}Every
commutative AG-groupoid is associative.\end{remark}

\noindent \begin{corollary}Every commutative AG-groupoid is a CA-AG-groupoid.\end{corollary}

In the following it is observed that subclass of Bol{*}-AG-groupoids
is distinct from that of CA-AG-groupoids and we have the following.
Recall that $S$ is called a Bol{*}-AG-groupoid if it satisfies the
identity $a(bc\cdot d)=(ab\cdot c)d\,\,\forall a,b,c,d\in S$.

\noindent \begin{theorem}\label{Thm: CA implies Bol*}Every CA-AG-groupoid
is Bol{*}-AG-groupoid.\end{theorem}\begin{proof}Let $S$ be a CA-AG-groupoid
and $a,b,c,d\in S$, then by using cyclic associativity, left invertive
law and medial law;
\begin{align*}
a(bc\cdot d) & =d(a\cdot bc)=d(c\cdot ab)=ab\cdot dc=c(ab\cdot d)=c(db\cdot a)\\
 & =a(c\cdot db)=db\cdot ac=da\cdot bc=c(da\cdot b)=c(ba\cdot d)\\
 & =d(c\cdot ba)=d(a\cdot cb)=d(b\cdot ac)=ac\cdot db=b(ac\cdot d)\\
 & =b(dc\cdot a)=a(b\cdot dc)=dc\cdot ab=(ab\cdot c)d\\
\Rightarrow a(bc\cdot d) & =(ab\cdot c)d.
\end{align*}
Hence $S$ is a Bol{*}-AG-groupoid.\end{proof}

\noindent Following example shows that the converse of the Theorem
\ref{Thm: CA implies Bol*}, is not true.

\noindent \begin{example}Consider the Bol{*}-AG-groupoid of order
$3$ given in Table $6$.\end{example}

\noindent \begin{center}%
\begin{tabular}{c|ccc}
$\cdot$ & $1$ & $2$ & $3$\tabularnewline
\hline 
$1$ & $1$ & $2$ & $3$\tabularnewline
$2$ & $3$ & $1$ & $2$\tabularnewline
$3$ & $2$ & $3$ & $1$\tabularnewline
\multicolumn{1}{c}{} &  &  & \tabularnewline
\multicolumn{4}{c}{Table $6$}\tabularnewline
\end{tabular}\end{center}

\noindent which is not a CA-AG-groupoid, because $1(2\cdot3)\neq3(1\cdot2)$.

\noindent Further, we have the following.

\noindent \begin{lemma}\label{lem:Bol*-AG-band is commutative}Every
Bol{*}-AG-band is commutative semigroup.\end{lemma}

\noindent \begin{proof}Let $S$ be a Bol{*}-AG-band and $\, a,b\in S$.
Then
\begin{align*}
ab & =(aa)b=(ba)a=(bb\cdot a)a=(ab\cdot b)a=a(bb\cdot a)=a(ab\cdot b)\\
 & =(aa\cdot b)b=bb\cdot aa=b^{2}a^{2}=ba\Rightarrow ab=ba.
\end{align*}
Thus $S$ is commutative and hence the result follows by Remark \ref{remark: commutativite AG-groupoid is associative}.\end{proof}

\noindent From Lemma \ref{lem:Bol*-AG-band is commutative} and Proposition
\ref{cor. comm semigp imply CA}, it is obtained;

\noindent \begin{corollary}Every Bol{*}-AG-band is CA-AG-groupoid.\end{corollary}

\noindent \begin{theorem}\label{thm:Every-CA-AG-band commutative}Every
CA-AG-band is commutative semigroup. \end{theorem}

\noindent \emph{Proof.} Let $S$ be a CA-AG-band and $a,b\in S.$
Then 
\begin{align*}
ab & =(aa)b=(ba)a=(bb\cdot a)a=(ab\cdot b)a\\
 & =ab\cdot ab=aa\cdot bb=b(aa\cdot b)=b(b\cdot aa)\\
 & =b(a\cdot ba)=ba\cdot ba=bb\cdot aa=ba.\qed
\end{align*}

\noindent As every CA-AG-groupoid is Bol{*}-AG-groupoid (by Theorem
\ref{Thm: CA implies Bol*}) and every Bol{*}-AG-groupoid is paramedial
\cite[Lemma 9]{m.shah phd thesis}, so the following corollary is
obvious:

\noindent \begin{corollary}\label{thm2: CA implies paramedial}Every
CA-AG-groupoid is paramedial.\end{corollary}

\noindent The following example shows that the converse of Corollary
\ref{thm2: CA implies paramedial} is not valid.

\noindent \begin{example}Table $7$ represents paramedial AG-groupoid
of order $3$, which is not a CA-AG-groupoid. as $c(cb)\neq b(cc)$.\end{example}

\noindent \begin{center}%
\begin{tabular}{c|ccc}
$\cdot$ & $a$ & $b$ & $c$\tabularnewline
\hline 
$a$ & $a$ & $a$ & $a$\tabularnewline
$b$ & $a$ & $a$ & $a$\tabularnewline
$c$ & $a$ & $b$ & $a$\tabularnewline
\multicolumn{1}{c}{} &  &  & \tabularnewline
\multicolumn{4}{c}{Table $7$}\tabularnewline
\end{tabular}\end{center}

It is proved in \cite{m.shah phd thesis} that every paramedial AG-band
is commutative semigroup. Also by Proposition \ref{cor. comm semigp imply CA},
every commutative semigroup is CA-AG-groupoid. Hence we have the following.

\begin{corollary}\label{cor:-paramedial-AG-band}Every paramedial
AG-band is CA-AG-groupoid.\end{corollary}

Now relation of CA-AG-groupoid with left, middle and right nuclear
square AG-groupoids is being studied in the following.

\begin{theorem}\label{thm3: ca-aggp is left nuclear sq.}\label{Every-CAag is right Nuc sq}Let
$S$ be a CA-AG-groupoid. Then the following hold:

\begin{enumerate}[(i)]\item $S$ is left nuclear square.

\item $S$ is right nuclear square.\end{enumerate}\end{theorem}

\begin{proof}Let $S$ be a CA-AG-groupoid and $a,b,c\in S$. Then

\begin{enumerate}[(i)]\item By Corollary \ref{thm2: CA implies paramedial}
and left invertive law, we have 
\[
a^{2}(bc)=(aa)(bc)=(ca)(ba)=(ba\cdot a)c=(aa\cdot b)c=(a^{2}b)c.
\]
Hence $S$ is left nuclear square.

\item By left invertive law, Part $(i)$, cyclic associativity and
Corollary \ref{thm2: CA implies paramedial}.
\begin{align*}
(ab)c^{2} & =(c^{2}b)a=c^{2}(ba)=a(c^{2}b)=a(cc\cdot b)=a(bc\cdot c)=c(a\cdot bc)\\
 & =bc\cdot ca=ac\cdot cb=b(ac\cdot c)=b(cc\cdot a)=a(b\cdot cc)=a(bc^{2}).
\end{align*}
\end{enumerate} \qquad{}Hence $S$ is right nuclear square.\end{proof}

Following example shows that neither a left nuclear square nor a right
nuclear square AG-groupoid is always a CA-AG-groupoid.

\begin{example}Left nuclear square AG-groupoid of order $3$ is given
in Table $8$, which is not a CA-AG-groupoid, because $3(3\cdot2)\neq2(3\cdot3)$.\end{example}

\noindent \begin{center}%
\begin{tabular}{c|ccc}
$\cdot$ & $1$ & $2$ & $3$\tabularnewline
\hline 
$1$ & $1$ & $1$ & $1$\tabularnewline
$2$ & $1$ & $1$ & $1$\tabularnewline
$3$ & $1$ & $2$ & $1$\tabularnewline
\multicolumn{1}{c}{} &  &  & \tabularnewline
\multicolumn{4}{c}{Table $8$}\tabularnewline
\end{tabular}\end{center}

\noindent In Table $9$ a right nuclear square AG-groupoid of order
$3$ is given which is not a CA-AG-groupoid because $3(3\cdot2)\neq2(3\cdot3)$.

\noindent \begin{center}
\begin{tabular}{c|ccc}
$\cdot$  & $1$  & $2$  & $3$\tabularnewline
\hline 
$1$  & $1$  & $1$  & $1$\tabularnewline
$2$  & $1$  & $1$  & $3$\tabularnewline
$3$  & $1$  & $2$  & $1$\tabularnewline
\multicolumn{1}{c}{} &  &  & \tabularnewline
\multicolumn{4}{c}{Table $9$}\tabularnewline
\end{tabular}
\par\end{center}

Next attention is paid towards middle nuclear square AG-groupoid.
Generally speaking no relation exist between middle nuclear square
AG-groupoids and CA-AG-groupoid. The following example shows that
neither every middle nuclear square AG-groupoid is CA-AG-groupoid,
nor every CA-AG-groupoid is middle nuclear square AG-groupoid .

\begin{example}The following AG-groupoid of order $3$ given in Table
$10$, is middle nuclear square AG-groupoid but not CA, because $c(cb)\neq b(cc)$.\end{example}

\noindent \begin{center}%
\begin{tabular}{c|ccc}
$\cdot$ & $a$ & $b$ & $c$\tabularnewline
\hline 
$a$ & $a$ & $a$ & $a$\tabularnewline
$b$ & $a$ & $a$ & $a$\tabularnewline
$c$ & $a$ & $b$ & $b$\tabularnewline
\multicolumn{1}{c}{} &  &  & \tabularnewline
\multicolumn{4}{c}{Table $10$}\tabularnewline
\end{tabular}\end{center}

CA-AG-groupoid of order $5$, given in Table $11$, is not a middle
nuclear square AG-groupoid since $5(5^{2}\cdot5)\neq(5\cdot5^{2})5$.

\noindent \begin{center}%
\begin{tabular}{c|ccccc}
$\cdot$ & $1$ & $2$ & $3$ & $4$ & $5$\tabularnewline
\hline 
$1$ & $1$ & $1$ & $1$ & $1$ & $1$\tabularnewline
$2$ & $1$ & $1$ & $1$ & $1$ & $1$\tabularnewline
$3$ & $1$ & $1$ & $1$ & $1$ & $2$\tabularnewline
$4$ & $1$ & $1$ & $1$ & $1$ & $2$\tabularnewline
$5$ & $1$ & $1$ & $1$ & $3$ & $4$\tabularnewline
\multicolumn{1}{c}{} &  &  &  &  & \tabularnewline
\multicolumn{6}{c}{Table $11$}\tabularnewline
\end{tabular}\end{center}

However, a left alternative CA-AG-groupoid is middle nuclear square.
As

\begin{theorem}\label{thm: left alternative CAgp imply middle nuclare square}Every
left alternative CA-AG-groupoid is middle nuclear square.\end{theorem}

\begin{proof}Let $S$ be a left alternative CA-AG-groupoid and $a,b,c\in S$.
Then by definition of left alternative, cyclic associativity, Corollary
\ref{thm2: CA implies paramedial} and left invertive law;
\begin{align*}
a(b^{2}c) & =a(bb\cdot c)=a(b\cdot bc)=bc\cdot ab=(ab\cdot c)b=(cb\cdot a)b\\
 & =ba\cdot cb=b(ba\cdot c)=b(ca\cdot b)=b(b\cdot ca)=ca\cdot bb\\
 & =(bb\cdot a)c=(b\cdot ba)c=(a\cdot bb)c=(ab^{2})c.
\end{align*}
Hence $S$ is middle nuclear square.\end{proof}

By coupling Theorem \ref{thm3: ca-aggp is left nuclear sq.} and Theorem
\ref{thm: left alternative CAgp imply middle nuclare square}, the
following corollary is obvious.

\begin{corollary}\label{cor:left alternate CAgp is nuclear square}Every
left alternative CA-AG-groupoid is nuclear square.\end{corollary}

\noindent As, for a right alternative AG-groupoid $S$ the following
conditions are equivalent \cite[Theorem 3.2]{on relation b/w right alternate=000026nuclear square imtiaz}.

\noindent $(i)$ $S$ is right nuclear square.

\noindent $(ii)$ $S$ is middle nuclear square.

\noindent $(iii)$ $S$ is nuclear square .\\
Using these results and Theorem \ref{Every-CAag is right Nuc sq},
we conclude that:

\begin{corollary}\label{cor: right alternative CAgp is nuclear square}Every
right alternative and hence alternative CA-AG-groupoid is nuclear
square.\end{corollary}

\section{CA-AG-GROUPOIDS, AG{*}-GROUPOIDS AND AG{*}{*}-GROUPOIDS\label{sec:CA-AG-groupoid,-AG*-groupoids-and}}

In this section it is seen that in general CA-AG-groupoid are neither
AG{*}-groupoids nor AG{*}{*}-groupoids. So the subclass, CA-AG-groupoid
is different from both of AG{*}-groupoids and AG{*}{*}-groupoids.

To begin with, consider the following; 

\begin{example}CA-AG-groupoid of order $4$, given in Table $12$,
is not an AG{*}-groupoid because $(aa)a\neq a(aa)$.\end{example}

\noindent \begin{center}
\begin{tabular}{c|cccc}
$\cdot$  & $a$  & $b$  & $c$  & $d$\tabularnewline
\hline 
$a$  & $b$  & $c$  & $c$  & $c$\tabularnewline
$b$  & $d$  & $c$  & $c$  & $c$\tabularnewline
$c$  & $c$  & $c$  & $c$  & $c$\tabularnewline
$d$  & $c$  & $c$  & $c$  & $c$\tabularnewline
\multicolumn{1}{c}{} &  &  &  & \tabularnewline
\multicolumn{5}{c}{Table $12$}\tabularnewline
\end{tabular}
\par\end{center}

\label{exp45 CA-AGgp not AG* nor AG**}CA-AG-groupoid of order $8$,
presented in Table $13$, is not an AG{*}{*}-groupoid because $3(2\cdot1)\neq2(3\cdot1)$

\noindent \begin{center}
\begin{tabular}{c|cccccccc}
$\cdot$ & $1$ & $2$ & $3$ & $4$ & $5$ & $6$ & $7$ & $8$\tabularnewline
\hline 
$1$ & $4$ & $4$ & $6$ & $4$ & $4$ & $4$ & $8$ & $4$\tabularnewline
$2$ & $5$ & $4$ & $4$ & $4$ & $4$ & $8$ & $4$ & $4$\tabularnewline
$3$ & $4$ & $7$ & $4$ & $4$ & $8$ & $4$ & $4$ & $4$\tabularnewline
$4$ & $4$ & $4$ & $4$ & $4$ & $4$ & $4$ & $4$ & $4$\tabularnewline
$5$ & $4$ & $4$ & $4$ & $4$ & $4$ & $4$ & $4$ & $4$\tabularnewline
$6$ & $4$ & $4$ & $4$ & $4$ & $4$ & $4$ & $4$ & $4$\tabularnewline
$7$ & $4$ & $4$ & $4$ & $4$ & $4$ & $4$ & $4$ & $4$\tabularnewline
$8$ & $4$ & $4$ & $4$ & $4$ & $4$ & $4$ & $4$ & $4$\tabularnewline
\multicolumn{1}{c}{} &  &  &  &  &  &  &  & \tabularnewline
\multicolumn{9}{c}{Table $13$}\tabularnewline
\end{tabular}
\par\end{center}

\noindent Further the following example establishes that there are
AG{*}{*}-groupoids which are not CA-AG-groupoid.

\begin{example} AG{*}{*}-groupoid of order $3$, given in Table $14$,
is not a CA-AG-groupoid because $2(2\cdot3)\neq3(2\cdot2)$.\end{example}

\noindent \begin{center}
\begin{tabular}{c|ccc}
$\cdot$ & $1$ & $2$ & $3$\tabularnewline
\hline 
$1$ & $1$ & $1$ & $1$\tabularnewline
$2$ & $1$ & $1$ & $3$\tabularnewline
$3$ & $1$ & $2$ & $1$\tabularnewline
\multicolumn{1}{c}{} &  &  & \tabularnewline
\multicolumn{4}{c}{Table $14$}\tabularnewline
\end{tabular}
\par\end{center}

However, we have the following;

\begin{theorem}Every right commutative AG{*}{*}-groupoid is a CA-AG-groupoid.\end{theorem}

\begin{proof}Let $S$ be a right commutative AG{*}{*}-groupoid and
$a,b,c\in S$. Using definition of right commutative and AG{*}{*}-groupoid,
we have $a(bc)=a(cb)=c(ab)$. Hence $S$ is a CA-AG-groupoid.\end{proof}

The following example is to show that above Theorem is true only for
AG{*}{*}-groupoid.

\noindent \begin{example}Right commutative AG-groupoid of order $3$,
given in Table $15$, is not a CA-AG-groupoid because $3(3\cdot1)\neq1(3\cdot3)$.\end{example}

\noindent \begin{center}
\begin{tabular}{c|ccc}
$\cdot$  & $1$  & $2$  & $3$\tabularnewline
\hline 
$1$  & $1$  & $1$  & $1$\tabularnewline
$2$  & $1$  & $1$  & $1$\tabularnewline
$3$  & $2$  & $2$  & $2$\tabularnewline
\multicolumn{1}{c}{} &  &  & \tabularnewline
\multicolumn{4}{c}{Table $15$}\tabularnewline
\end{tabular}
\par\end{center}

The following example is to show that there are right commutative
AG-groupoids which are not AG{*}{*}-groupoids

\begin{example}Table $16$, represents a right commutative AG-groupoid
of order $3$, which is not an AG{*}{*}-groupoid because $a(cb)\neq c(ab)$.\end{example}

\noindent \begin{center}
\begin{tabular}{c|ccc}
$\cdot$ & $a$ & $b$ & $c$\tabularnewline
\hline 
$a$ & $a$ & $a$ & $a$\tabularnewline
$b$ & $a$ & $a$ & $a$\tabularnewline
$c$ & $b$ & $b$ & $b$\tabularnewline
\multicolumn{1}{c}{} &  &  & \tabularnewline
\multicolumn{4}{c}{Table $16$}\tabularnewline
\end{tabular}
\par\end{center}

The following example is to show that neither every CA-AG-groupoid
is right commutative, nor every AG{*}{*} is right commutative.

\begin{example}\label{exp53 CA not right commute}\label{exp54 AG** not right commute}Table
13, given above represents a CA-AG-groupoid of order $8$. As $1(2\cdot3)\neq1(3\cdot2)$,
so it is not a right commutative AG-groupoid.

Table $17$ represents an AG{*}{*}-groupoid of order $3$, which is
not right commutative because $3(2\cdot3)\neq3(3\cdot2)$.

\noindent \begin{center}%
\begin{tabular}{c|ccc}
$\cdot$ & $1$ & 2 & $3$\tabularnewline
\hline 
$1$ & $1$ & $1$ & $1$\tabularnewline
$2$ & $1$ & $1$ & $3$\tabularnewline
$3$ & $1$ & $2$ & $1$\tabularnewline
\multicolumn{1}{c}{} &  &  & \tabularnewline
\multicolumn{4}{c}{Table $17$}\tabularnewline
\end{tabular}\end{center}\end{example}

However, we have the following;

\begin{theorem}Let $S$ be a CA-AG-groupoid then $S$ is right commutative
if and only if $S$ is AG{*}{*}.\end{theorem}\begin{proof}Let $S$
be a CA-AG-groupoid and $a,b,c\in S$. Suppose $S$ is right commutative,
then, by using definition of right commutative and cyclic associativity,
we have $a(bc)=a(cb)=b(ac)$. Hence $S$ is AG{*}{*}.

\quad{}Conversely, let $S$ be an AG{*}{*}-groupoid, then by cyclic
associativity and definition of AG{*}{*}, we have
\[
a(bc)=c(ab)=a(cb)=b(ac)=c(ba)=a(cb)\Rightarrow a(bc)=a(cb).
\]
Hence $S$ is right commutative.\end{proof}

As every AG{*}{*} having a cancellative element is T$^{1}$-AG-groupoid
\cite{9} and hence is T$^{3}$-AG-groupoid \cite{9}, so we have
the following corollary.

\begin{corollary}Every right commutative CA-AG-groupoid having a
cancellative element is a T$^{1}$-AG-groupoid and hence a T$^{3}$-AG-groupoid.\end{corollary}

Since every T$^{1}$-AG-groupoid is AG{*}{*} \cite[Theorem 3.4]{on relation b/w right alternate=000026nuclear square imtiaz},
thus;

\begin{corollary}CA-T$^{1}$-AG-groupoid is right commutative.\end{corollary}

\begin{theorem}\label{cor: ag-band imply CA}An AG{*}-band is CA.\end{theorem}\begin{proof}Let
$S$ be an AG{*}-band and $a,b,c\in S$. Then
\begin{align*}
a(bc) & =(ba)c=(ca)b=(ca)(bb)=(cb)(ab)=(a\cdot cb)b=(b\cdot cb)a\\
 & =(cb\cdot b)a=(bb\cdot c)a=(bc)a=(ac)b=c(ab)\Rightarrow a(bc)=c(ab).
\end{align*}
Hence $S$ is CA-AG-groupoid.\end{proof}

Following example verifies that the converse of Theorem \ref{cor: ag-band imply CA},
is not valid.

\begin{example}Table $18$, represents a CA-AG-groupoid of order
$4$, which is not AG{*}-band, because $(4\cdot4)4\neq4(4\cdot4)$
and $2\cdot2\neq2$.\end{example}

\noindent \begin{center}
\begin{tabular}{c|cccc}
$\cdot$ & $1$ & $2$ & $3$ & $4$\tabularnewline
\hline 
$1$ & $1$ & $1$ & $1$ & $1$\tabularnewline
$2$ & $1$ & $1$ & $1$ & $1$\tabularnewline
$3$ & $1$ & $1$ & $1$ & $2$\tabularnewline
$4$ & $1$ & $1$ & $1$ & $3$\tabularnewline
\multicolumn{1}{c}{} &  &  &  & \tabularnewline
\multicolumn{5}{c}{Table $18$}\tabularnewline
\end{tabular}
\par\end{center}

\begin{lemma}\label{thm:An-AG*-groupoid idemp is semigp}An AG{*}-band
is a semigroup.\end{lemma}

\begin{proof}Let $S$ be an AG{*}-band and $a,b,c\in S.$ Then
\begin{align*}
a(bc) & =(ba)c=(ba)(cc)=(bc)(ac)=(a\cdot bc)c=(c\cdot bc)a\\
 & =(bc\cdot c)a=(cc\cdot b)a=(cb)a=(ab)c\Rightarrow a(bc)=(ab)c.
\end{align*}
Hence $S$ is a semigroup.\end{proof}Since every AG{*}{*}-band is
a commutative semigroup \cite[Lemma 20]{m.shah phd thesis} and every
AG{*}{*}-$3$-band is a commutative semigroup \cite{AG**}, also by
Proposition \ref{cor. comm semigp imply CA}, every commutative semigroup
is CA-AG-groupoid. Hence we have the following corollary.

\begin{corollary}\label{cor: AG**-band and 3-band are CA-AGgp}$(i)$
Every AG{*}{*}-band is a CA-AG-groupoid.

~~~~~~~~~~~~~~~~~ $(ii)$ Every AG{*}{*}-$3$-band
is a CA-AG-groupoid.\end{corollary}As every AG-monoid is AG{*}{*}
\cite{AG**}, so from Corollary \ref{cor: AG**-band and 3-band are CA-AGgp}
the following is obvious;

\begin{corollary}Every AG-band having left identity is CA-AG-groupoid.\end{corollary}

\begin{theorem}\label{thm:50}Every CA-AG-$3$-band is a commutative
semigroup.\end{theorem}\begin{proof}Let $S$ be a CA-AG-$3$-band.
Then $\forall\, a,b\in S,$ by definition of AG-3-band, left invertive
law and Corollary \ref{thm2: CA implies paramedial};
\begin{eqnarray*}
ab & = & (aa\cdot a)b=(ba)(aa)=((bb\cdot b)a)(aa)\\
 & = & (aa\cdot a)(bb\cdot b)=(aa\cdot bb)(a\cdot b)=(b\cdot bb)(a\cdot aa)\\
\Rightarrow ab & = & ba.
\end{eqnarray*}
Thus $S$ is commutative and hence is a commutative semigroup.\end{proof}

\begin{remark}Let denote the set of all idempotents of an AG-groupoid
$S$ by $E(S)$.\end{remark}

\begin{lemma}For a CA-AG-groupoid $S$, if $E(S)\neq\phi$ then $E(S)$
is a semi-lattice.\end{lemma}

\begin{proof}Let $S$ be a CA-AG-groupoid and $a,\, b\in E(S)$.
Then
\begin{align*}
ab & =(aa)(bb)=(ab)(ab)=b(ab\cdot a)=a(b\cdot ab)\\
 & =a(b\cdot ba)=a(a\cdot bb)=(bb)(aa)=ba.
\end{align*}
Thus $E(S)$ is commutative. Hence $E(S)$ is a commutative semigroup
of idempotents and thus is a semi-lattice.\end{proof}

\begin{lemma}Let $S$ be a CA-AG-groupoid such that $E(S)\neq\phi$.
Then $e$ is the identity of $eS$ and $Se$ for every $e\in E(S)$.\end{lemma}

\begin{proof}\begin{enumerate}[(i)]\item Using cyclic associativity,
left invertive law, and Corollary \ref{thm2: CA implies paramedial};{\small{}
\begin{align*}
e(eS) & =\underset{a\in S}{\cup}e(ea)=\underset{a\in S}{\cup}(ee)(ea)=\underset{a\in S}{\cup}(ae)(ee)\\
 & =\underset{a\in S}{\cup}(ae)e=\underset{a\in S}{\cup}(ee)a=\underset{a\in S}{\cup}ea=eS.
\end{align*}
}Thus $e$ is the left identity for $eS$. Also by left invertive
law, medial law, cyclic associativity and Corollary \ref{thm2: CA implies paramedial};{\small{}
\begin{align*}
(eS)e & =\underset{a\in S}{\cup}(ea)e=\underset{a\in S}{\cup}(ee\cdot a)e=\underset{a\in S}{\cup}(ea)(ee)=\underset{a\in S}{\cup}(ee)(ae)\\
 & =\underset{a\in S}{\cup}e(ae)=\underset{a\in S}{\cup}e(ea)=\underset{a\in S}{\cup}(ee)(ea)=\underset{a\in S}{\cup}(ae)(ee)\\
 & =\underset{a\in S}{\cup}(ae)e=\underset{a\in S}{\cup}(ee)a=\underset{a\in S}{\cup}ea=eS.
\end{align*}
}Thus $e$ is also right identity of $eS$ and hence the identity
of $eS$.

\item By using cyclic associativity; 
\[
e(Se)=\underset{a\in S}{\cup}e(ae)=\underset{a\in S}{\cup}e(ea)=\underset{a\in S}{\cup}a(ee)=\underset{a\in S}{\cup}ae=Se.
\]
Thus $e$ is the left identity for $Se$. Also, by  Corollary \ref{thm2: CA implies paramedial}
and cyclic associativity;{\small{}
\begin{align*}
(Se)e & =\underset{a\in S}{\cup}(ae)e=\underset{a\in S}{\cup}(ae)(ee)=\underset{a\in S}{\cup}(ee)(ea)\\
 & =\underset{a\in S}{\cup}e(ea)=\underset{a\in S}{\cup}a(ee)=\underset{a\in S}{\cup}ae=Se.
\end{align*}
}Thus $e$ is also the right identity for $Se$. Hence $e$ is the
identity for $Se$.\end{enumerate}\end{proof}

The following example shows that there exist left commutative AG-groupoids
which are not CA-AG-groupoids.

\begin{example}Left commutative AG-groupoid of order $3$ is presented
in Table 19, which is not a CA-AG-groupoid because $3(2\cdot1)\neq1(3\cdot2)$.\end{example}

\noindent \begin{center}
\begin{tabular}{c|ccc}
$\cdot$  & $1$  & 2  & $3$\tabularnewline
\hline 
$1$  & $1$  & $1$  & $1$\tabularnewline
$2$  & $1$  & $1$  & $1$\tabularnewline
$3$  & $2$  & $2$  & $1$\tabularnewline
\multicolumn{1}{c}{} &  &  & \tabularnewline
\multicolumn{4}{c}{Table $19$}\tabularnewline
\end{tabular}
\par\end{center}

Next, we have the following;

\begin{theorem}Every left commutative AG{*}-groupoid is a CA-AG-groupoid.\end{theorem}

\begin{proof}Let $S$ be a left commutative AG{*}-groupoid and $a,b,c\in S$,
then by definition of AG{*}, left commutative and left invertive law;
\[
a(bc)=(ba)c=(ca)b=(ac)b=c(ab)\Rightarrow a(bc)=c(ab).
\]
Hence $S$ is CA-AG-groupoid.\end{proof}

Following example shows that a left commutative CA-AG-groupoid may
not be an AG{*}-groupoid.

\begin{example}Left commutative CA-AG-groupoid of order $4$, is
given in Table $20$, which is not AG{*}-groupoid because $(d\cdot d)d\neq d(d\cdot d)$.\end{example}

\noindent \begin{center}
\begin{tabular}{c|cccc}
$\cdot$ & $a$ & $b$ & $c$ & $d$\tabularnewline
\hline 
$a$ & $a$ & $a$ & $a$ & $a$\tabularnewline
$b$ & $a$ & $a$ & $a$ & $a$\tabularnewline
$c$ & $a$ & $a$ & $a$ & $b$\tabularnewline
$d$ & $a$ & $a$ & $a$ & $c$\tabularnewline
\multicolumn{1}{c}{} &  &  &  & \tabularnewline
\multicolumn{5}{c}{Table $20$}\tabularnewline
\end{tabular}
\par\end{center}

\begin{theorem}\label{thm:Every-CA-AG*-groupoid-is commute}Let $S$
be a CA-AG{*}-groupoid. Then the following holds.

$(i)$ $S$ is bi-commutative.

$(ii)$ \label{lem:Every-CA-AG*-groupoid-is semigp}$S$ is semigroup.\end{theorem}

\begin{proof}\begin{enumerate}[(i)]\item Let $S$ be a CA-AG{*}-groupoid
and $a,b,c\in S$, then
\begin{eqnarray*}
(ab)c & = & (cb)a=b(ca)=a(bc)=c(ab)=(ac)b=(bc)a\\
 & = & c(ba)=a(cb)=(ca)b=(ba)c\Rightarrow(ab)c=(ba)c.
\end{eqnarray*}
Hence $S$ is left commutative. Again
\begin{align*}
a(bc) & =c(ab)=b(ca)=(cb)a=(ab)c=b(ac)\\
 & =c(ba)=a(cb)\Rightarrow a(bc)=a(cb).
\end{align*}
Thus $S$ is also right commutative and hence $S$ is bi-commutative.

\item Let $S$ be a CA-AG{*}-groupoid and $a,b,c\in S.$ Then 
\begin{eqnarray*}
a(bc) & = & c(ab)=b(ca)=(cb)a=(ab)c\Rightarrow a(bc)=(ab)c.
\end{eqnarray*}
\end{enumerate}

\quad{} Hence $S$ is a semigroup.\end{proof}

The following example is to show that there exist Bi-commutative AG-groupoids
which are neither CA-AG-groupoids nor AG{*}-groupoids.

\begin{example}Table $21$ represents a Bi-commutative AG-groupoid
of order $4$, which is neither CA-AG-groupoid nor AG{*}-groupoid.\end{example}

\noindent \begin{center}%
\begin{tabular}{c|cccc}
$\cdot$ & $1$ & $2$ & $3$ & $4$\tabularnewline
\hline 
$1$ & $1$ & $1$ & $3$ & $3$\tabularnewline
$2$ & $1$ & $1$ & $4$ & $4$\tabularnewline
$3$ & $3$ & $3$ & $1$ & $1$\tabularnewline
$4$ & $3$ & $3$ & $1$ & $1$\tabularnewline
\multicolumn{1}{c}{} &  &  &  & \tabularnewline
\multicolumn{5}{c}{Table $21$}\tabularnewline
\end{tabular}\end{center}

\noindent Next example shows that there exist CA-AG{*}-groupoids which
are not Bi-commutative AG-groupoid.

\noindent \begin{example}\label{exp: CA-AG*-gp not commutative-1}Table
$22$ shows a CA-AG{*}-groupoid of order $4$, which is not commutative.\end{example}

\noindent \begin{center}%
\begin{tabular}{c|cccc}
$\cdot$ & $1$ & $2$ & $3$ & $4$\tabularnewline
\hline 
$1$ & $3$ & $3$ & $3$ & $3$\tabularnewline
$2$ & $4$ & $3$ & $3$ & $3$\tabularnewline
$3$ & $3$ & $3$ & $3$ & $3$\tabularnewline
$4$ & $3$ & $3$ & $3$ & $3$\tabularnewline
\multicolumn{1}{c}{} &  &  &  & \tabularnewline
\multicolumn{5}{c}{Table $22$}\tabularnewline
\end{tabular}\end{center}

\noindent From the above discussion we conclude that:

\begin{corollary}CA-AG{*}-groupoid is a non-commutative semigroup.\end{corollary}Since
every T$^{1}$-AG-3-band is AG{*}-groupoid \cite{17}, so from Theorem
\ref{thm:Every-CA-AG*-groupoid-is commute}, we have the following
corollary.\begin{corollary}$(i)$ Every CA-T$^{1}$-AG-$3$-band
is a bi-commutative AG-groupoid.

~~~~~~~~~~~~~~~~~~~$(ii)$ Every CA-T$^{1}$-AG-$3$-band
is a semigroup.\end{corollary}

The relationship among different subclasses of AG-groupoids is given
in the following Venn Diagram.\begin{center}

\noindent \end{center}

\end{document}